\documentclass[a4paper,11pt,reqno,twoside]{amsart}
\usepackage{Kloosterman_distri}
\usepackage{float}
\floatstyle{boxed}
\newfloat{Fig}{htbp}{Fig}[section]
\floatname{Fig}{Fig}
\begin{document}%
\title[Distribution of short sums of classical Kloosterman sums of prime powers moduli]{Distribution of short sums of classical Kloosterman sums of prime powers moduli}
\author[G. Ricotta]{Guillaume Ricotta}
\address{Universit\'{e} de Bordeaux \\
Institut de Math\'{e}matiques de Bordeaux \\
351, cours de la Lib\'{e}ration \\
33405 Talence cedex \\
France}
\email{Guillaume.Ricotta@math.u-bordeaux.fr}
%
%
\date{Version of \today} 
\begin{abstract}
In \cite{Pe-Ge1}, the author proved, under some very general conditions, that short sums of $\ell$-adic trace functions over finite fields of varying center converges in law to a Gaussian random variable or vector. The main inputs are P.~Deligne's equidistribution theorem, N.~Katz' works and the results surveyed in \cite{MR3338119}. In particular, this applies to $2$-dimensional Kloosterman sums $\mathsf{Kl}_{2,\mathbb{F}_q}$ studied by N.~Katz in \cite{MR955052} and in \cite{MR1081536} when the field $\mathbb{F}_q$ gets large.
\par
This article considers the case of short sums of normalized classical Kloosterman sums of prime powers moduli $\mathsf{Kl}_{p^n}$, as $p$ tends to infinity among the prime numbers and $n\geq 2$ is a fixed integer. A convergence in law towards a real-valued standard Gaussian random variable is proved under some very natural conditions.
\end{abstract}
\maketitle
\begin{center}
\textit{In memory of Prince Rogers Nelson and David Robert Jones. Enjoy your new career in your new purple town.}
\end{center}
\tableofcontents
\section{Introduction and statement of the results}%
Let $p$ be an odd prime number. For $\mathbb{F}_q$ the finite field of cardinality $q$ and of characteristic $p$, $t_q$ a complex-valued function on $\mathbb{F}_q$ and $I_q$ a subset of $\mathbb{F}_q$, the normalized partial sum of $t_q$ over $I_q$ is defined by
\begin{equation*}
S(t_q,I_q)\coloneqq\frac{1}{\sqrt{\abs{I_q}}}\sum_{x\in I_q}t_q(x).
\end{equation*}
where as usual $\abs{I_q}$ stands for the cardinality of $I_q$. Such sums have a long history in analytic number theory, confer \cite[Chapter 12]{IwKo}. The normalization is explained by the fact that in a number theory context one expects the square-root cancellation philosophy. One can define a complex-valued random variable on $\mathbb{F}_q$ endowed with the uniform measure by
\begin{equation*}
\forall x\in\mathbb{F}_q,\quad S(t_q,I_q;x)\coloneqq S(t_q,I_q+x)
\end{equation*}
where as usual $I_q+x$ stands for the translate of $I_q$ by $x$ for any $x$ in $\mathbb{F}_q$.
\par
Given a sequence $t_q$ of $\ell$-adic trace functions over $\mathbb{F}_q$ and a sequence $I_q$ of subsets of $\mathbb{F}_q$, C.~Perret-Gentil got interested in \cite{Pe-Ge1} in the distribution as $q$ and $\abs{I_q}$ tend to infinity of the sequence of complex-valued random variables $S(t_q,I_q;\ast)$ and proved a deep general result under very natural conditions. Let us mention that his general result is not only a generalization but also an improvement over previous works such as \cite{MR0055368}, \cite{MR2756007}, \cite{MR3063909} and \cite{MR1652005}.
\par
Let us state the case of the normalized Kloosterman sums of rank $2$ given by
\begin{equation*}
\forall x\in\mathbb{F}_q,\quad t_q(x)=\mathsf{Kl}_{2,\mathbb{F}_q}(x)\coloneqq\frac{-1}{\sqrt{q}}\sum_{\substack{(x_1,x_2)\in\mathbb{F}_q^\times\times\mathbb{F}_q^\times \\
x_1x_2=x}}e\left(\frac{\text{Tr}_{\mathbb{F}_q\vert\mathbb{F}_p}(x_1+x_2)}{p}\right)\in\mathbb{R}
\end{equation*}
where as usual $e(z)\coloneqq\exp{(2i\pi z)}$ for any complex number $z$.
\par
C.~Perret-Gentil proved the following qualitative result.
\begin{theorem}[C.~Perret-Gentil (Qualitative result)]\label{theo_qual_CPG}
As $q$ and $\abs{I_q}$ tend to infinity with $\log{(\abs{I_q})}=o(\log{(q)})$ then the sequence of real-valued random variables $S(\mathsf{Kl}_{2,\mathbb{F}_q},I_q;\ast)$ converges in law to a real-valued standard Gaussian random variable.
\end{theorem}
He also proved the following quantitative result.
\begin{theorem}[C.~Perret-Gentil (Quantitative result)]\label{theo_quant_CPG}
As $q$ and $\abs{I_q}$ tend to infinity with $\log{(\abs{I_q})}=o(\log{(q)})$ then
\begin{multline*}
\frac{\left\vert\left\{x\in\mathbb{F}_q, \alpha\leq S(\mathsf{Kl}_{2,\mathbb{F}_q},I_q;x)\leq\beta\right\}\right\vert}{q}=\frac{1}{\sqrt{2\pi}}\int_{\alpha}^\beta\exp{\left(\frac{-x^2}{2}\right)}\mathrm{d}x \\
+O_\epsilon\left((\beta-\alpha)\left(q^{-1/2+\epsilon}+\left(\frac{\log{(\abs{I_q})}}{\log{(q)}}\right)^{2/5}+\frac{1}{\sqrt{\abs{I_q}}}\right)\right)
\end{multline*}
for any real numbers $\alpha<\beta$ and for any $0<\epsilon<1/2$.
\end{theorem}
The main purpose of this work is to consider the case of Kloosterman sums of prime powers moduli, namely to replace finite fields by finite rings, and to give a probabilistic meaning to the histogram given in Figure \ref{fig_K2(p=61,n=2)}.
\par
The normalized Kloosterman sum of modulus $p^n$ is the real number given by
\begin{equation*}
\mathsf{Kl}_{p^n}(a)\coloneqq\frac{1}{p^{n/2}}S\left(a,1;p^n\right)=\frac{1}{p^{n/2}}\sum_{\substack{1\leq x\leq p^n \\
p\nmid x}}\
e\left(\frac{ax+\overline{x}}{p^n}\right)
\end{equation*}
for any integer $a$ and where as usual $\overline{x}$ stands for the inverse of $x$ modulo $p^n$.
\par
For any subset $I_{p^n}$ of $\left(\Z/p^n\Z\right)^\times$, let
\begin{equation*}
S\left(\mathsf{Kl}_{p^n},I_{p^n}\right)\coloneqq\frac{1}{\sqrt{\abs{I_{p^n}}}}\sum_{x\in I_{p^n}}\mathsf{Kl}_{p^n}(x)
\end{equation*}
be the normalized partial sum over $I_{p^n}$.
\par
Given a sequence of sets $I_{p^n}$ of $\Z/p^n\Z$, we are interested in the distribution of the sequence of real random variables over $\left(\Z/p^n\Z\right)^\times$ endowed with the uniform measure given by
\begin{equation*}
\forall x\in\left(\Z/p^n\Z\right)^\times,\quad S\left(\mathsf{Kl}_{p^n},I_{p^n};x\right)\coloneqq S\left(\mathsf{Kl}_{p^n},I_{p^n}+x\right).
\end{equation*}
\begin{Fig}
\begin{center}
\includegraphics[scale=0.6,angle=0]{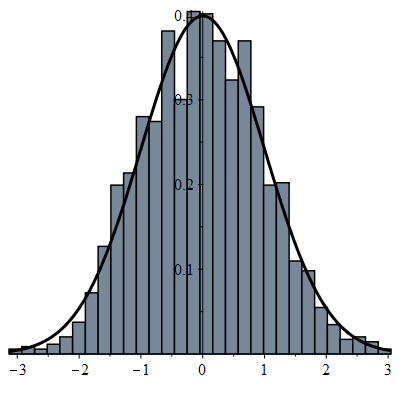}
\end{center}
\caption{Distribution of $S\left(\mathsf{Kl}_{41^2},I_{41^2};\ast\right)$, namely $p=41$ and $n=2$, for a set $I_{41^2}$ of cardinality $29$. In bold, the density function of a standard Gaussian random variable.}
\label{fig_K2(p=61,n=2)}
\end{Fig}
\strut\newline Let us state the qualitative result of this work.
\begin{theoint}[Qualitative result]\label{theo_A}
Let $n\geq 2$ be a fixed integer. Assume that
\begin{equation}\label{eq_cond_fundaa}
\forall (x,y)\in I_{p^n}\times I_{p^n},\quad x\neq y\Rightarrow p\nmid x-y. 
\end{equation}
for any prime number $p$. If $p$ and $\abs{I_{p^n}}$ tend to infinity with 
\begin{equation}\label{eq_cond_funda_2}
\log{\left(\abs{I_{p^n}}\right)}=o\left(\log{(p)}\right)
\end{equation}
then the sequence of real-valued random variables $S\left(\mathsf{Kl}_{p^n},I_{p^n};\ast\right)$ converges in law to a standard Gaussian real-valued random variable.
\end{theoint}
\begin{remark}
This theorem is the analogue of Theorem \ref{theo_qual_CPG}. The condition \eqref{eq_cond_fundaa} is new and comes from the context of finite rings in this work instead of finite fields in \cite{Pe-Ge1} whereas the condition \eqref{eq_cond_funda_2} is exactly the same and is inherent to the method of proof itself namely the method of moments. Note that the condition \eqref{eq_cond_fundaa} requires that $\abs{I_{p^n}}<p$ holds, which is automatically satisfied by \eqref{eq_cond_funda_2}.
\end{remark}
\newpage
Let us state the quantitative result of this work.
\begin{theoint}[Quantitative result]\label{theo_B}
Let $n\geq 2$ be a fixed integer and
\begin{equation*}
\beta_n\coloneqq\begin{cases}
1/2 & \text{if $\;\;2\leq n\leq 5$,} \\
\frac{4(n-1)}{2^n} & \text{otherwise}.
\end{cases}
\end{equation*}
Assume that
\begin{equation}\label{eq_cond_funda_1}
\forall (x,y)\in I_{p^n}\times I_{p^n},\quad x\neq y\Rightarrow p\nmid x-y. 
\end{equation}
for any prime number $p$. If $p$ and $\abs{I_{p^n}}$ tend to infinity with 
\begin{equation}\label{eq_cond_funda_22}
\log{\left(\abs{I_{p^n}}\right)}=o\left(\log{(p)}\right)
\end{equation}
then
\begin{multline*}
\frac{\left\vert\left\{x\in\left(\Z/p^n\Z\right)^\times, \alpha\leq S\left(\mathsf{Kl}_{p^n},I_{p^n};x\right)\leq\beta\right\}\right\vert}{\varphi\left(p^n\right)}=\frac{1}{\sqrt{2\pi}}\int_{\alpha}^\beta\exp{\left(\frac{-x^2}{2}\right)}\mathrm{d}x \\
+O_\epsilon\left(\max{\left(\frac{1}{\abs{I_{p^n}}},\left(\frac{\log{\left(\abs{I_{p^n}}\right)}}{\log{(p)}}\right)^{3/4}\right)}+p^{-\beta_n+3\epsilon}+\frac{\beta-\alpha}{\sqrt{\abs{I_{p^n}}}}\right)
\end{multline*}
for any real numbers $\alpha<\beta$ and for any $0<\epsilon<\beta_n/3$.
\end{theoint}
\begin{remark}
Once again, this theorem is the perfect analogue of Theorem \ref{theo_quant_CPG}.
\end{remark}
\noindent{\textbf{Organization of the paper. }}%
The main tool involved in Theorem \ref{theo_A} is recalled in Subsection \ref{sec_A}. The technical results required in Theorem \ref{theo_B} are stated in Subsection \ref{sec_B}. Theorem \ref{theo_A} is proved in Section \ref{sec_process}. The proof of Theorem \ref{theo_B} is given in Section \ref{sec_moments}.
\begin{notations}
The main parameter in this paper is an odd prime number $p$, which tends to infinity. Thus, if $f$ and $g$ are some $\C$-valued function of the real variable then the notations $f(p)=O_A(g(p))$ or $f(p)\ll_A g(p)$ mean that $\abs{f(p)}$ is smaller than a "constant", which only depends on $A$, times $g(p)$ at least for $p$ large enough.
\par
$n\geq 2$ is a fixed integer.
\par
For any real number $x$ and integer $k$, $e_k(x)\coloneqq\exp{\left(\frac{2i\pi x}{k}\right)}$.
\par
For any finite set $S$, $\vert S\vert$ stands for its cardinality. 
\par
We will denote by $\epsilon$ an absolute positive constant whose definition may change from one line to the next one.
\par
The notation $\sideset{}{^\times}\sum$ means that the summation is over a set of integers coprime with $p$.
\par
Finally, if $\mathcal{P}$ is a property then $\delta_{\mathcal{P}}$ is the Kronecker symbol, namely $1$ if $\mathcal{P}$ is satisfied and $0$ otherwise.
\end{notations}
\begin{merci}%
The main structure of this paper was worked out while the author was visiting the Republic of Cameroon in December 2016 and January 2017. He would like to heartily thank all the wonderful people he met during his journey.
\par
The author is financed by the ANR Project Flair ANR-17-CE40-0012.
\par
Last but not least, the author would like to thank the anonymous referee and Corentin~Perret-Gentil for their relevant comments.
\end{merci}
\section{The main ingredients}\label{sec_path}%
\subsection{Moments of products of additively shifted Kloosterman sums}\label{sec_A}%
The crucial ingredient in the proof of Theorem \ref{theo_A} is the asymptotic evaluation of the complete sums of products of shifted Kloosterman sums $\mathsf{S}_{p^n}(\uple{\mu})$ defined by
\begin{equation}\label{eq_Spmu}
\mathsf{S}_{p^n}(\uple{\mu})\coloneqq\frac{1}{\varphi(p^n)}\sum_{a\in\left(\mathbb{Z}/p^n\mathbb{Z}\right)^\times}\prod_{\tau\in \mathbb{Z}/p^n\mathbb{Z}}\mathsf{Kl}_{p^n}(a+\tau)^{\mu(\tau)}
\end{equation}
for $\uple{\mu}=\left(\mu(\tau)\right)_{\tau\in\mathbb{Z}/p^n\mathbb{Z}}$ a sequence of $p^n$-tuples of non-negative integers different from the $0$-tuple.
\par
Let us define for such sequence $\uple{\mu}$,
\begin{eqnarray*}
\mathsf{T}(\uple{\mu}) & \coloneqq & \left\{\tau\in\Z/p^n\Z, \mu(\tau)\geq 1\right\}\subset\Z/p^n\Z, \\
\overline{\mathsf{T}}(\uple{\mu}) & \coloneqq & \left\{\tau\bmod{p}, \tau\in\mathsf{T}(\uple{\mu})\right\}\subset\Z/p\Z
\end{eqnarray*}
and
\begin{equation}\label{eq_apmu}
\mathsf{A}_{p^n}(\uple{\mu})\coloneqq\left\{a\in\left(\Z/p^n\Z\right)^\times, \forall\tau\in\mathsf{T}(\uple{\mu}), a+\tau\in\left(\left(\Z/p^n\Z\right)^\times\right)^2\right\}.
\end{equation}
\par
The following proposition, which contains an asymptotic formula for the sums $\mathsf{S}_{p^n}(\uple{\mu})$, is an improvement of \cite[Proposition 4.10]{MR3854900} in the sense that the dependency in the tuple $\uple{\mu}$ in the error term has been made explicit.
\begin{proposition}\label{propo_shifted_Kl}
Let $\uple{\mu}=\left(\mu(\tau)\right)_{\tau\in\mathbb{Z}/p^n\mathbb{Z}}$ be a sequence of $p^n$-tuples of non-negative integers satisfying
\begin{equation}\label{eq_assump_mu}
\sum_{\tau\in\Z/p^n\Z}\mu(\tau)\leq M
\end{equation}
for some absolute positive constant $M$. If 
\begin{equation}\label{eq_assump_p}
p>\max{\left(M,2n-5\right)}
\end{equation}
then
\begin{multline}\label{eq_asymp_S}
\mathsf{S}_{p^n}(\uple{\mu})=\left[\prod_{\tau\in \mathbb{Z}/p^n\mathbb{Z}}\delta_{2\mid\mu(\tau)}\binom{\mu(\tau)}{\mu(\tau)/2}\right]\frac{\left\vert\mathsf{A}_{p^n}(\uple{\mu})\right\vert}{\varphi(p^n)} \\
+O_{\epsilon}\left(2^{\sum_{\tau\in\mathsf{T}(\uple{\mu})}\mu(\tau)}\left(p^{-\frac{4(n-1)}{2^n}+\epsilon}+\frac{\vert\mathsf{T}(\uple{\mu})\vert\times 2^{\vert\mathsf{T}(\uple{\mu})\vert}}{p}\right)\right)
\end{multline}
for any $\epsilon>0$ and where the implied constant only depends on $\epsilon$.
\end{proposition}
The dependency in the tuple $\uple{\mu}$ in \cite[Proposition 4.7]{MR3854900} also has to be made explicit. Let us recall some additional notations, which coincide exactly with the notations used in \cite{MR3854900} and whose motivations can be found in this reference. Let $\mathsf{B}_{p^n}(\uple{\mu})$ be the subset of the $\abs{\mathsf{T}(\uple{\mu})}$-tuples $\uple{b}=\left(b_\tau\right)_{\tau\in\mathsf{T}(\uple{\mu})}$ of integers in $\{1,\dots,(p-1)/2\}$ satisfying
\begin{equation}\label{eq_asump_1}
\forall(\tau,\tau')\in\mathsf{T}(\uple{\mu})^2,\quad b_\tau^2-\tau\equiv b_{\tau'}^2-\tau'\bmod{p}
\end{equation}
and
\begin{equation}\label{eq_deuxieme}
\forall\tau\in\mathsf{T}(\uple{\mu}),\quad p\nmid b_\tau^2-\tau.
\end{equation}
Let $\uple{\ell}=\left(\ell_\tau\right)_{\tau\in\mathsf{T}(\uple{\mu})}$ be a $\left\vert\mathsf{T}(\uple{\mu})\right\vert$-tuple of integers. For any integer $j$ in $\{1,\dots,n-1\}$, let us define
\begin{equation}\label{eq_mmm}
m_{\uple{b},\uple{\ell}}(j,j)=\sum_{\tau\in\mathsf{T}(\uple{\mu})}\ell_\tau\overline{b_\tau}^{2j-1}
\end{equation}
and the following associated object
\begin{equation}\label{eq_tricky_counting}
\mathsf{N}(\uple{\mu},\uple{\ell};w)\coloneqq\sum_{\substack{\uple{b}\in\mathsf{B}_{p^n}(\uple{\mu}) \\
m_{\uple{b},\uple{\ell}}(1,1)\equiv w\bmod{p} \\
\forall j\in\{2,\dots,n-1\},\quad m_{\uple{b},\uple{\ell}}(j,j)\equiv 0\bmod{p}}}1
\end{equation}
for any $w$ modulo $p$.
\begin{lemma}\label{propo_tricky_counting}
Let $\uple{\mu}=\left(\mu(\tau)\right)_{\tau\in\mathbb{Z}/p^n\mathbb{Z}}$ be a sequence of $p^n$-tuples of non-negative integers satisfying
$\abs{\mathsf{T}(\uple{\mu})}=\abs{\overline{\mathsf{T}}(\uple{\mu})}$ and $\uple{\ell}$ be a $\abs{\mathsf{T}(\uple{\mu})}$-tuple of integers satisfying
\begin{equation*}
\forall\tau\in\mathsf{T}(\uple{\mu}),\quad\abs{\ell_\tau}<p
\end{equation*}
and $\uple{\ell}\neq\uple{0}$. One uniformly has
\begin{equation*}
\mathsf{N}(\uple{\mu},\uple{\ell};w)\ll\vert\mathsf{T}(\uple{\mu})\vert\times 2^{\vert\mathsf{T}(\uple{\mu})\vert}
\end{equation*}
for any $w\bmod{p}$ where the implied constant is absolute.
\end{lemma}
\begin{proof}[\proofname{} of lemma \ref{propo_tricky_counting}]%
Let us briefly indicate the required changes in the proof of \cite[Proposition 4.7]{MR3854900}. Let $k\coloneqq\vert\mathsf{T}(\uple{\mu})\vert$ for simplicity. On the one hand, if $(p,w)=1$ then the polynomial $\psi(R_{\uple{\ell}}(\uple{Y};w))$ in $\mathbb{F}_p[Z]$ defined in \cite[Page 15]{MR3854900} is of degree exactly $k2^{k-1}$ and admits at most $k2^{k-1}$ roots. On the other hand, if $w\equiv 0\pmod{p}$ then the non-zero polynomial $\psi(S_{\uple{\ell}}(\uple{Y}))$ in $\mathbb{F}_p[Z]$ defined in \cite[Page 507]{MR3854900} is of degree at most $(k-1)2^{k-2}$ and admits at most $(k-1)2^{k-2}$ roots.
\end{proof}
Let us give the proof of Proposition \ref{propo_shifted_Kl}.
\begin{proof}[\proofname{} of proposition \ref{propo_shifted_Kl}]%
By \cite[Page 511]{MR3854900}, the error term to bound is given by
\begin{multline*}
\mathsf{Err}_{p^n}(\uple{\mu})\coloneqq\frac{1}{\varphi(p^n)}\sum_{\uple{b}\in\mathsf{B}_{p^n}(\uple{\mu})}\prod_{\tau\in\mathsf{T}(\uple{\mu})}\left(\frac{b_\tau}{p^n}\right)^{\mu(\tau)} \\
\sum_{\substack{a\in\Z/p^n\Z \\
\forall\tau\in\mathsf{T}(\uple{\mu}), a\equiv b_\tau^2-\tau\bmod{p}}}\prod_{\tau\in\mathsf{T}(\uple{\mu})}\;\;\sideset{}{^\circ}\sum_{0\leq u_\tau\leq\mu(\tau)}\binom{\mu(\tau)}{u_\tau}\cos{\left[\left(\mu(\tau)-2u_\tau\right)\left(\frac{4\pi s_{a+\tau,p^n}}{p^n}+\theta_{p^n}\right)\right]}
\end{multline*}
where $\sideset{}{^\circ}\sum$  means that the summation is over the $u_\tau $'s satisfying
\begin{equation*}
\exists\tau_0\in\mathsf{T}(\uple{\mu}),\quad\mu(\tau_0)-2u_{\tau_0}\neq 0.
\end{equation*}
In the previous equation $s_{a+\tau,p^n}$ stands for any square-root modulo $p^n$ of $a+\tau$ for any relevant $a$ and $\tau$.
\par
Obviously,
\begin{multline*}
\mathsf{Err}_{p^n}(\uple{\mu})=\frac{1}{\varphi(p^n)}\sum_{\uple{b}\in\mathsf{B}_{p^n}(\uple{\mu})}\prod_{\tau\in\mathsf{T}(\uple{\mu})}\left(\frac{b_\tau}{p^n}\right)^{\mu(\tau)}\;\;\sideset{}{^\circ}\sum_{\substack{\uple{u}=\left(u_\tau\right)_{\tau\in\mathsf{T}(\uple{\mu})} \\
\forall\tau\in\mathsf{T}(\uple{\mu}), 0\leq u_\tau\leq\mu(\tau)}}\prod_{\tau\in\mathsf{T}(\uple{\mu})}\binom{\mu(\tau)}{u_\tau} \\
\sum_{\substack{a\in\Z/p^n\Z \\
\forall\tau\in\mathsf{T}(\uple{\mu}), a\equiv b_\tau^2-\tau\bmod{p}}}\prod_{\tau\in\mathsf{T}(\uple{\mu})}\cos{\left[\left(\mu(\tau)-2u_\tau\right)\left(\frac{4\pi s_{a+\tau,p^n}}{p^n}+\theta_{p^n}\right)\right]}.
\end{multline*}
By Euler's formula,
\begin{multline}\label{eq_estim_fin_1}
\left\vert\mathsf{Err}_{p^n}(\uple{\mu})\right\vert\leq\;\;\sideset{}{^\circ}\sum_{\substack{\uple{u}=\left(u_\tau\right)_{\tau\in\mathsf{T}(\uple{\mu})} \\
\forall\tau\in\mathsf{T}(\uple{\mu}), 0\leq u_\tau\leq\mu(\tau)}}\prod_{\tau\in\mathsf{T}(\uple{\mu})}\binom{\mu(\tau)}{u_\tau}\frac{1}{2^{\vert\mathsf{T}(\uple{\mu})\vert}}\sum_{J\subset\mathsf{T}(\uple{\mu})}\frac{1}{\varphi(p^n)}\sum_{\uple{b}\in\mathsf{B}_{p^n}(\uple{\mu})} \\
\left\vert\sum_{\substack{a\in\Z/p^n\Z \\
\forall\tau\in\mathsf{T}(\uple{\mu}), a\equiv b_\tau^2-\tau\bmod{p}}}e_{p^n}\left(\sum_{\tau\in J}\left(\mu(\tau)-2u_\tau\right)s_{a+\tau,p^n}-\sum_{\tau\in J^c}\left(\mu(\tau)-2u_\tau\right)s_{a+\tau,p^n}\right)\right\vert.
\end{multline}
\par
Let us define
\begin{equation*}
\mathsf{Err}_{p^n}(\uple{\mu},\uple{\ell})\coloneqq\frac{1}{\varphi(p^n)}\sum_{\uple{b}\in\mathsf{B}_{p^n}(\uple{\mu})}\left\vert\sum_{\substack{a\in\Z/p^n\Z \\
\forall\tau\in\mathsf{T}(\uple{\mu}), a\equiv b_\tau^2-\tau\bmod{p}}}e_{p^n}\left(\sum_{\tau\in\mathsf{T}(\uple{\mu})}\ell_\tau s_{a+\tau,p^n}\right)\right\vert
\end{equation*}
for any $\abs{\mathsf{T}(\uple{\mu})}$-tuple $\uple{\ell}$ of integers satisfying
\begin{equation*}
\uple{\ell}\in\prod_{\tau\in\mathsf{T}(\uple{\mu})}[-\mu(\tau),\mu(\tau)]\quad\text{ and }\quad \uple{\ell}\neq\uple{0}.
\end{equation*}
By \cite[Equation (4.37)]{MR3854900},
\begin{multline*}
\mathsf{Err}_{p^n}(\uple{\mu},\uple{\ell})\ll_{\epsilon}p^{-\frac{4(n-1)}{2^n}+\epsilon}+\frac{\mathsf{N}(\uple{\mu},\uple{\ell};0)}{p}+\sum_{k=1}^{n-1}\frac{1}{p^k}\sum_{\substack{v\bmod{p^{n-k}} \\
(p,v)=1}}\frac{1}{\abs{v}}\mathsf{N}\left(\uple{\mu},\uple{\ell};\overline{c_1^\prime}vp^{k-1}\right)
\end{multline*}
for any $\epsilon>0$ and for some integer $c_1^\prime$ coprime with $p$ defined in \cite[Lemma 4.6]{MR3854900}, $\overline{c_1^\prime}$ being its inverse modulo $p$.
\par
By Lemma \ref{propo_tricky_counting}, one gets
\begin{equation}\label{eq_estim_fin_2}
\mathsf{Err}_{p^n}(\uple{\mu},\uple{\ell})\ll_{\epsilon}p^{-\frac{4(n-1)}{2^n}+\epsilon}+\frac{\vert\mathsf{T}(\uple{\mu})\vert\times 2^{\vert\mathsf{T}(\uple{\mu})\vert}}{p}
\end{equation}
for any $\epsilon>0$.
\par
By \eqref{eq_estim_fin_1} and \eqref{eq_estim_fin_2},
\begin{equation*}
\mathsf{Err}_{p^n}(\uple{\mu})\ll_\epsilon 2^{\sum_{\tau\in\mathsf{T}(\uple{\mu})}\mu(\tau)}\left(p^{-\frac{4(n-1)}{2^n}+\epsilon}+\frac{\vert\mathsf{T}(\uple{\mu})\vert\times 2^{\vert\mathsf{T}(\uple{\mu})\vert}}{p}\right)
\end{equation*}
for any $\epsilon>0$.
\end{proof}
The following proposition, which heavily relies on A.~Weil's proof of the Riemann hypothesis for curves over finite fields and is \cite[Proposition 4.8]{MR3854900}, states an asymptotic formula for the cardinality of the sets $\mathsf{A}_{p^n}(\uple{\mu})$.
\begin{proposition}[G.~Ricotta-E.~Royer]\label{propo_cardinality}
Let $\uple{\mu}=\left(\mu(\tau)\right)_{\tau\in\mathbb{Z}/p^n\mathbb{Z}}$ be a sequence of $p^n$-tuples of non-negative integers. If $p$ is odd then
\begin{equation}\label{eq_cardinality}
\left\vert\mathsf{A}_{p^n}(\uple{\mu})\right\vert=\frac{\varphi(p^n)}{2^{\abs{\overline{\mathsf{T}}(\uple{\mu})}}}\left(1+O\left(\frac{2^{\abs{\overline{\mathsf{T}}(\uple{\mu})}}\abs{\overline{\mathsf{T}}(\uple{\mu})}}{p^{1/2}}\right)\right)
\end{equation}
where the implied constant is absolute.
\end{proposition}
\subsection{Various approximation results}\label{sec_B}%
The following lemma, which enables us to approximate characteristic functions of random variables from their moments, is a reformulation of \cite[Lemma 5.1]{Pe-Ge1}.
\begin{lemma}\label{lemma_charac}
Let $X_1$ and $X_2$ be real-valued random variables. If
\begin{equation*}
\mathbb{E}\left(X_1^k\right)=\mathbb{E}\left(X_2^k\right)+O\left(h(k)\right)
\end{equation*}
for any non-negative integer $k$ and for some function $h:\mathbb{R}\to\mathbb{R}$ then
\begin{equation*}
\mathbb{E}\left(e^{iuX_1}\right)=\mathbb{E}\left(e^{iuX_2}\right)+O\left(\frac{\abs{u}^k}{k!}\left\vert\mathbb{E}\left(X_2^{k/2}\right)\right\vert+\left(1+\abs{u}^k\right)\max_{\ell<k}{\left(\abs{h(\ell)}\right)}\right)
\end{equation*}
for any even integer $k\geq 1$ and any real number $u$. 
\end{lemma}
The following lemma, which allows us to approximate joint distributions of random variables via their characteristic functions, follows from \cite[Section 4]{MR3091515}.
\begin{lemma}\label{lemma_joint}
Let $X_1$ and $X_2$ be real-valued random variables and $\alpha<\beta$ be real numbers. If
\begin{equation*}
\mathbb{E}\left(e^{2i\pi uX_1}\right)=\mathbb{E}\left(e^{2i\pi uX_2}\right)+O\left(g(\abs{u})\right)
\end{equation*}
for any real number $u$ and some continuous function $g:\mathbb{R}\to\R_+$ then
\begin{equation*}
\mathbb{P}\left(X_1\in[\alpha,\beta]\right)=\mathbb{P}\left(X_2\in[\alpha,\beta]\right)+O\left(\left(1+\frac{1}{t}\right)\int_0^tg(u)\mathrm{d}u+\frac{1}{t}\int_0^t\left\vert\mathbb{E}\left(e^{2i\pi uX_1}\right)\right\vert\mathrm{d}u\right)
\end{equation*}
for any real number $t>0$.
\end{lemma}
Finally, the following lemma is an explicit version of the Berry-Esseen theorem in dimension one (see \cite[Theorem 13.2]{MR855460}).
\begin{lemma}\label{lemma_berry}
Let $\alpha<\beta$ be two real numbers. Let $X_1,\dots,X_h$ be centered independent identically distributed real-valued random variables of variance $1$ satisfying $\mathbb{E}\left(\left\vert X_1\right\vert^3\right)<\infty$ and
\begin{equation*}
S_H=\frac{X_1+\dots+ X_H}{\sqrt{H}}.
\end{equation*}
One has
\begin{equation*}
\mathbb{P}\left(S_H\in[\alpha,\beta]\right)=P(X\in[\alpha,\beta])+O\left(\frac{\beta-\alpha}{\sqrt{H}}\right)
\end{equation*}
for any standard Gaussian real-valued random variable $X$.
\end{lemma}
\section{Proof of the qualitative result (Theorem \ref{theo_A})}\label{sec_process}%
\subsection{Asymptotic expansion of the moments}%
The $k$'th moment of the real-valued random variable $S\left(\mathsf{Kl}_{p^n},I_{p^n};\ast\right)$ is defined by
\begin{equation*}
M_k\left(\mathsf{Kl}_{p^n},I_{p^n}\right)\coloneqq\frac{1}{\varphi\left(p^n\right)}\;\;\sideset{}{^\times}\sum_{x\bmod{p^n}}S\left(\mathsf{Kl}_{p^n},I_{p^n};x\right)^k
\end{equation*}
for any non-negative integer $k$.
\par
Let $\left(U_h\right)_{h\geq 1}$ be a sequence of independent identically distributed random variables of probability law $\mu$ given by
\begin{equation*}
\mu=\frac{1}{2}\delta_0+\mu_1
\end{equation*}
for the Dirac measure $\delta_0$ at $0$ and
\begin{equation*}
\mu_1(f)=\frac{1}{2\pi}\int_{-2}^2\frac{f(x)\mathrm{d}x}{\sqrt{4-x^2}}
\end{equation*}
for any real-valued continuous function $f$ on $[-2,2]$ and let
\begin{equation}\label{eq_SH}
S_H=\frac{U_1+\dots+U_H}{\sqrt{H}}.
\end{equation}
\par
The following proposition is an asymptotic expansion of these moments.
\begin{proposition}\label{propo_moment}%
Let $n\geq 2$ be a fixed integer. Assume that
\begin{equation}\label{eq_cond_funda}
\forall (x,y)\in I_{p^n}\times I_{p^n},\quad x\neq y\Rightarrow p\nmid x-y
\end{equation}
for any prime number $p$. If $p>\max{(k,2n-5)}$ then
\begin{equation*}
M_k\left(\mathsf{Kl}_{p^n},I_{p^n}\right)=\mathbb{E}\left(S_H^k\right)+O_{\epsilon}\left(4^k\left(\frac{H^{k/2+1}}{\sqrt{p}}+\frac{H^{k/2}}{p^{\frac{4(n-1)}{2^n}-\epsilon}}\right)\right)
\end{equation*}
for any $\epsilon>0$ and where the implied constant only depends on $\epsilon$.
\end{proposition}
\begin{proof}[\proofname{} of proposition \ref{propo_moment}]%
Let us fix a non-negative integer $k$ and let us set
\begin{equation*}
I_{p^n}=\left\{a_1,\dots,a_{H}\right\}\subset\Z/p^n\Z
\end{equation*}
where $H\coloneqq\left\vert I_{p^n}\right\vert$. Obviously, $H$ depends on $p$ and $n$ but such dependency has been removed for clarity. With these notations,
\begin{equation*}
M_k\left(\mathsf{Kl}_{p^n},I_{p^n}\right)=\frac{1}{H^{k/2}}\frac{1}{\varphi\left(p^n\right)}\;\;\sideset{}{^\times}\sum_{x\bmod{p^n}}\left(\sum_{i=1}^H\mathsf{Kl}_{p^n}(a_i+x)\right)^k.
\end{equation*}
\par
By the multinomial formula,
\begin{align*}
M_k\left(\mathsf{Kl}_{p^n},I_{p^n}\right) & =\frac{1}{H^{k/2}}\sum_{\substack{\uple{k}=(k_1,\dots,k_H)\in\Z_+^H \\
k_1+\dots+k_H=k}}\binom{k}{k_1,\dots,k_H}\frac{1}{\varphi\left(p^n\right)}\;\;\sideset{}{^\times}\sum_{x\bmod{p^n}}\prod_{i=1}^H\mathsf{Kl}_{p^n}(a_i+x)^{k_i} \\
& =\frac{1}{H^{k/2}}\sum_{\substack{\uple{k}=(k_1,\dots,k_H)\in\Z_+^H \\
k_1+\dots+k_H=k}}\binom{k}{k_1,\dots,k_H}\mathsf{S}_{p^n}(\uple{\mu}_{\uple{k}})
\end{align*}
where
\begin{equation*}
\mu_\uple{k}(\tau)=\begin{cases}
k_i & \text{if $\;\exists i\in\{1,\dots,H\}, \tau=a_i$,} \\
0 & \text{otherwise}
\end{cases}
\end{equation*}
for any $\tau$ in $\Z/p^n\Z$.
\par
By Proposition \ref{propo_shifted_Kl} and Proposition \ref{propo_cardinality}, if $p>\max{(k,2n-5)}$ then
\begin{multline}\label{eq_revert}
M_k\left(\mathsf{Kl}_{p^n},I_{p^n}\right)=\frac{1}{H^{k/2}}\sum_{\substack{\uple{k}=(k_1,\dots,k_H)\in\Z_+^H \\
k_1+\dots+k_H=k}}\binom{k}{k_1,\dots,k_H}\left[\prod_{i=1}^H\delta_{2\mid k_i}\binom{k_i}{k_i/2}\right]\frac{1}{2^{\abs{\mathsf{T}(\uple{\mu}_\uple{k})}}} \\
+O_{\epsilon}\left(4^k\left(\frac{H^{k/2+1}}{\sqrt{p}}+\frac{H^{k/2}}{p^{\frac{4(n-1)}{2^n}-\epsilon}}\right)\right)
\end{multline}
for any $\epsilon>0$ since $\overline{\mathsf{T}}(\uple{\mu}_\uple{k})=\mathsf{T}(\uple{\mu}_\uple{k})$ by \eqref{eq_cond_funda}. The obvious fact that
\begin{equation*}
\left\vert\mathsf{T}(\uple{\mu}_\uple{k})\right\vert\leq\min{\left(H,k\right)}
\end{equation*}
has been used.
\par
One has
\begin{equation*}
M_k\left(\mathsf{Kl}_{p^n},I_{p^n}\right)=\mathbb{E}\left(S_H^k\right)+O_{\epsilon}\left(4^k\left(\frac{H^{k/2+1}}{\sqrt{p}}+\frac{H^{k/2}}{p^{\frac{4(n-1)}{2^n}-\epsilon}}\right)\right)
\end{equation*}
for any $\epsilon>0$ and where $S_H$ is defined in \eqref{eq_SH} and since
\begin{equation*}
\mathbb{E}\left(U_1^m\right)=\begin{cases}
1 & \text{if $m=0$,} \\
\frac{\delta_{2\mid m}}{2}\binom{m}{m/2} & \text{if $m\geq 1$}
\end{cases}
\end{equation*}
by \cite[Equation (3.1)]{MR3854900}
\end{proof}
\subsection{Proof of Theorem \ref{theo_A}}%
In order to prove Theorem \ref{theo_A}, it is enough to prove that, for any non-negative integer $k$, the $k$'th moment of the real-valued random variable $S\left(\mathsf{Kl}_{p^n},I_{p^n};\ast\right)$ converges to the the $k$'th moment of a real-valued standard Gaussian random variable by \cite[Section 5.8.4]{MR2125120}.
\par
Let us fix a non-negative integer $k$. By Proposition \ref{propo_moment}, if $p>\max{(k,2n-5)}$ then
\begin{equation*}
M_k\left(\mathsf{Kl}_{p^n},I_{p^n}\right)=\mathbb{E}\left(S_H^k\right)+O_{\epsilon}\left(4^k\left(\frac{H^{k/2+1}}{\sqrt{p}}+\frac{H^{k/2}}{p^{\frac{4(n-1)}{2^n}-\epsilon}}\right)\right)
\end{equation*}
for any $\epsilon>0$ where $H\coloneqq\left\vert I_{p^n}\right\vert$ and $S_H$ is defined in \eqref{eq_SH}.
\par
By the central limit theorem, the random variable $S_H$ converges in law as $H$ tends to infinity to a real-valued standard Gaussian random variable $U$. The random variable $S_H$ being uniformly integrable by \cite[Chapter 5.5]{MR2125120}, one has
\begin{equation}\label{eq_lim_mom}
\lim_{H\to+\infty}\mathbb{E}\left(S_H^k\right)=\mathbb{E}\left(U^k\right)
\end{equation}
by \cite[Theorem 7.5.1]{MR2125120}.
\par
Finally,
\begin{equation*}
\lim_{p,H\to+\infty}M_k\left(\mathsf{Kl}_{p^n},I_{p^n}\right)=\mathbb{E}\left(U^k\right)
\end{equation*}
by \eqref{eq_lim_mom} in the regime given in \eqref{eq_cond_funda_2}, as desired.
\section{Proof of the quantitative result (Theorem \ref{theo_B})}\label{sec_moments}%
\subsection{Bounds for the moments of the probabilistic model}%
The following proposition contains bounds for the moments of the random variable $S_H$ defined in \eqref{eq_SH}.
\begin{proposition}\label{propo_bSH}
Let $k$ be any non-negative integer. One has $\mathbb{E}\left(S_H^k\right)=0$ if $k$ is odd and
\begin{equation*}
\mathbb{E}\left(S_H^k\right)\ll\frac{k!}{(k/2)!}
\end{equation*}
if $k$ is even.
\end{proposition}
\begin{remark}
As explained in the proof of Theorem \ref{theo_A}, $\mathbb{E}\left(S_H^k\right)$ converges to
\begin{equation*}
\delta_{2\mid k}\frac{k!}{2^{k/2}(k/2)!}
\end{equation*}
as $H$ tends to infinity. Thus, the bound given in Proposition \ref{propo_bSH} is close from the truth and is sufficient for our purposes.
\end{remark}
\begin{remark}
Corentin~Perret-Gentil mentioned that this result is hidden in \cite{Pe-Ge1} in a more theoretical language.
\end{remark}
\begin{proof}[\proofname{} of proposition \ref{propo_bSH}]%
By \eqref{eq_revert},
\begin{equation*}
\mathbb{E}\left(S_H^{k}\right)=\frac{1}{H^{k/2}}\sum_{\substack{\uple{k}=(k_1,\dots,k_H)\in\Z_+^H \\
k_1+\dots+k_H=k}}\binom{k}{k_1,\dots,k_H}\left[\prod_{i=1}^H\delta_{2\mid k_i}\binom{k_i}{k_i/2}\right]\frac{1}{2^{\abs{\mathsf{T}(\uple{\mu}_\uple{k})}}}
\end{equation*}
The $k$'th moment vanishes if $k$ is odd. Let us assume from now on that $k$ is even, in which case
\begin{equation*}
\mathbb{E}\left(S_H^{k}\right)\leq\frac{1}{H^{k/2}}\frac{k!}{(k/2)!}\sum_{\substack{\uple{\ell}=(\ell_1,\dots,\ell_H)\in\Z_+^H \\
\ell_1+\dots+\ell_H=k/2}}\frac{(k/2)!}{\ell_1!\dots\ell_H!}=\frac{k!}{(k/2)!}.
\end{equation*}
\end{proof}
\subsection{Proof of Theorem \ref{theo_B}}%
We follow essentially the method of proof of Theorem \ref{theo_quant_CPG}. Let $H=\left\vert I_{p^n}\right\vert$. Firstly, note that
\begin{equation*}
\frac{H}{\sqrt{p}}+\frac{1}{p^{\frac{4(n-1)}{2^n}}}\ll\frac{H^{\alpha_n}}{p^{\beta_n}}
\end{equation*}
where
\begin{equation*}
(\alpha_n,\beta_n)\coloneqq\begin{cases}
(1,1/2) & \text{if $\;\;2\leq n\leq 5$,} \\
\left(0,\frac{4(n-1)}{2^n}\right) & \text{otherwise}
\end{cases}
\end{equation*}
since
\begin{equation*}
\frac{4(n-1)}{2^n}\geq\frac{1}{2}\quad\text{ if and only if }\quad 2\leq n\leq 5
\end{equation*}
and by \eqref{eq_cond_funda_22}.
\par
Let us fix $0<\epsilon<\beta_n/3$ and let $k$ be an even integer suitably chosen later and satisfying
\begin{equation*}
2\alpha_n\leq k\leq\epsilon\frac{\log{(p)}}{\log{(4H)}}\quad\text{ and }\quad k\to +\infty,
\end{equation*}
which is possible by \eqref{eq_cond_funda_22}.
\par
By Proposition \ref{propo_moment},
\begin{equation}\label{eq_mom}
M_k\left(\mathsf{Kl}_{p^n},I_{p^n}\right)=\mathbb{E}\left(S_H^k\right)+O_{\epsilon}\left(p^{-\beta_n+2\epsilon}\right)
\end{equation}
where $S_H$ is defined in \eqref{eq_SH}.
\par
Let us denote by $\Phi_{p^n}$ the characteristic function of $S\left(\mathsf{Kl}_{p^n},I_{p^n};\ast\right)$ and by $\Phi_H$ the characteristic function of $S_H$. By Lemma \ref{lemma_charac} and \eqref{eq_mom},
\begin{equation}\label{eq_charac}
\Phi_{p^n}(u)=\Phi_H(u)+O_\epsilon\left(\frac{\abs{u}^k}{k!}\left\vert\mathbb{E}\left(S_H^{k/2}\right)\right\vert+p^{-\beta_n+2\epsilon}\left(1+\abs{u}^k\right)\right)
\end{equation}
for any real number $u$.
\par
Let $\alpha<\beta$ be two real numbers and $t\geq 1$ be a real number determined later. By Lemma \ref{lemma_joint} and \eqref{eq_charac}, one gets
\begin{multline}\label{eq_joint}
\mathbb{P}\left(\left\{x\in\left(\Z/p^n\Z\right)^\times, \alpha\leq S\left(\mathsf{Kl}_{p^n},I_{p^n};x\right)\leq\beta\right\}\right)=\mathbb{P}\left(S_H\in[\alpha,\beta]\right) \\
+O\left(\int_0^tg(u)\mathrm{d}u+\frac{1}{t}\int_0^t\left\vert\Phi_H(2\pi u)\right\vert\mathrm{d}u\right)
\end{multline}
where
\begin{equation*}
g(u)\coloneqq\left(\frac{(2\pi u)^k}{k!}\left\vert\mathbb{E}\left(S_H^{k/2}\right)\right\vert+p^{-\beta_n+2\epsilon}\left(1+(2\pi u)^k\right)\right)
\end{equation*}
for any non-negative real number $u$.
\par
Let us bound the second error term in \eqref{eq_joint}. By the independence of the random variables $U_1$, $\dots$, $U_H$,
\begin{equation*}
\Phi_H(2\pi u)=\left(\mathbb{E}\left(e^{\frac{2i\pi u}{\sqrt{H}}U_1}\right)\right)^H
\end{equation*}
for any real number $u$. The random variable $U_1$ being $4$-subgaussian, since it is centered and bounded by $2$ (see \cite[Page 11]{RRS} and \cite[Proposition B.6.2]{Ko}), it turns out that
\begin{equation*}
\mathbb{E}\left(e^{\frac{2i\pi u}{\sqrt{H}}U_1}\right)\ll e^{-8\pi^2u^2/H}
\end{equation*}
for any real number $u$. Thus, the second error term in \eqref{eq_joint} satisfies
\begin{equation}\label{eq_b2}
\frac{1}{t}\int_0^t\left\vert\Phi_H(2\pi u)\right\vert\mathrm{d}u\ll\frac{1}{t}.
\end{equation}
\par
The first error term in \eqref{eq_joint} is trivially bounded by
\begin{equation*}
(2\pi)^k\frac{t^{k+1}}{(k+1)!}\left\vert\mathbb{E}\left(S_H^{k/2}\right)\right\vert+p^{-\beta_n+2\epsilon}t\left(1+\frac{(2\pi t)^{k}}{k+1}\right).
\end{equation*}
\par
By Proposition \ref{propo_bSH},
\begin{align*}
(2\pi)^k\frac{t^{k+1}}{(k+1)!}\left\vert\mathbb{E}\left(S_H^{k/2}\right)\right\vert & \leq(2\pi t)^{k+1}\frac{(k/2)!}{(k+1)!(k/4)!} \\
& \ll\left(2\pi e^{3/4}t\right)^{k+1}k^{-3k/4}
\end{align*}
by Stirling's formula. Let us choose
\begin{equation*}
t=\frac{k^\gamma}{2\pi e^{3/4}}
\end{equation*}
where $\gamma=\gamma(k)>0$ will be chosen later. Thus, the first error term in \eqref{eq_joint} is bounded by
\begin{equation}\label{eq_b1}
\ll k^{\gamma(k+1)-3k/4}+p^{-\beta_n+2\epsilon}k^{\gamma(k+1)}.
\end{equation}
\par
By \eqref{eq_b1} and \eqref{eq_b2}, 
\begin{multline*}
\mathbb{P}\left(\left\{x\in\left(\Z/p^n\Z\right)^\times, \alpha\leq S\left(\mathsf{Kl}_{p^n},I_{p^n};x\right)\leq\beta\right\}\right)=\mathbb{P}\left(S_H\in[\alpha,\beta]\right) \\
+O_\epsilon\left(k^{-\gamma}+k^{\gamma(k+1)-3k/4}+p^{-\beta_n+2\epsilon}k^{\gamma(k+1)}\right).
\end{multline*}
\par
Let us choose
\begin{equation*}
\gamma=\gamma(k)=\frac{3k}{4(k+1)}
\end{equation*}
such that
\begin{multline*}
\mathbb{P}\left(\left\{x\in\left(\Z/p^n\Z\right)^\times, \alpha\leq S\left(\mathsf{Kl}_{p^n},I_{p^n};x\right)\leq\beta\right\}\right)=\mathbb{P}\left(S_H\in[\alpha,\beta]\right) \\
+O_\epsilon\left(k^{-3/4}+p^{-\beta_n+2\epsilon}k^{3k/4}\right).
\end{multline*}
\par
Let us choose
\begin{equation*}
k=\min{\left(H^{4/3},\epsilon\frac{\log{(p)}}{\log{(4H)}}\right)}\to+\infty
\end{equation*}
such that
\begin{equation*}
k^{3k/4}=e^{\frac{3}{4}k\log{(k)}}\leq e^{\epsilon\frac{\log{(p)}}{\log{(4H)}}\log{(H)}}\leq p^\epsilon
\end{equation*}
and
\begin{multline}\label{eq_joint_fin}
\mathbb{P}\left(\left\{x\in\left(\Z/p^n\Z\right)^\times, \alpha\leq S\left(\mathsf{Kl}_{p^n},I_{p^n};x\right)\leq\beta\right\}\right)=\mathbb{P}\left(S_H\in[\alpha,\beta]\right) \\
+O_\epsilon\left(\max{\left(\frac{1}{H},\left(\frac{\log{(H)}}{\log{(p)}}\right)^{3/4}\right)}+p^{-\beta_n+3\epsilon}\right).
\end{multline}
\par
Theorem \ref{theo_B} is implied by \eqref{eq_joint_fin} and Lemma \ref{lemma_berry}.
\bibliographystyle{alpha}
\bibliography{biblio}
\end{document}